\documentclass[12pt]{article}
\usepackage[utf8]{inputenc}
\usepackage[english]{babel}
\usepackage{amsmath,amsthm,amssymb, xcolor}
\usepackage[hidelinks]{hyperref}
\usepackage{enumitem}

\usepackage[margin=2.5cm]{geometry}

\DeclareMathOperator{\per}{PER}
\DeclareMathOperator{\dist}{dist}
\DeclareMathOperator{\proj}{proj}

\newtheorem{theorem}{Theorem}
\newtheorem{lemma}{Lemma}
\newtheorem{definition}{Definition}
\newtheorem{prop}{Proposition}
\newtheorem{corollary}{Corollary}[theorem]

\newtheorem{remark}{Remark}

\newtheoremstyle{named}{}{}{\itshape}{}{\bfseries}{.}{.5em}{\thmnote{#3's }#1}
\theoremstyle{named}

\title{Strong greedoid structure of $r$-removed $P$-orderings}

\author{Dmitrii Krachun\footnotemark[1]\footnote{Princeton University} \, and\, Rozalina Mirgalimova\footnotemark[2]\footnote{St. Petersburg State University}}

\begin{document}
\maketitle
\begin{abstract}
    Inspired by the notion of 
    \emph{$r$-removed $P$-orderings} introduced in the setting of Dedekind domains by Bhargava \cite{Bha09-1} we study its generalization in the framework of arbitrary (generalised) ultrametric spaces. We show that sets of maximal "$r$-removed perimeter" can be constructed by a greedy algorithm and form a strong greedoid. This gives a simplified proof of several theorems in \cite{Bha09-1} and also generalises the results of \cite{GP21} which considered the case $r=0$ corresponding, in turn, to simple $P$-orderings of \cite{Bha97}.
\end{abstract}

\section{Introduction}

Motivated by questions in polynomial function theory, Bhargava \cite{Bha97} introduced the notion of \emph{$P$-orderings} for a subset $X$ of a Dedekind domain $D$. The construction is as follows. Given a prime ideal $P\subset D$, let $a_0$ be an arbitrary element of $X$ and for $k=1,2,\dots$ choose $a_k\in X$  to minimize
\[
\nu_P\left((a_k-a_0)(a_k-a_1)\dots (a_k-a_{k-1})\right),
\]
where $\nu_P$ denotes the $P$-adic valuation on $D$. One of the results of \cite{Bha97} is the surprising fact that, despite the fact that typically the choice of each $a_k$ is non-unique, the sequence of the resulting valuations does not depend on the specific choice of $\{a_i\}$ but only on $X$ and $P$. Later, to study basis of the ring of polynomials with integer-valued divided differences,  Bhargava \cite{Bha09-1} generalised this construction to \emph{$r$-removed $P$-orderings}. For an $r$-removed $P$-ordering one again chooses a sequence $\{a_i\}$ of elements from $X$ but now the first $r+1$ elements $a_0,\dots, a_r$ are chosen arbitrary and then each new element minimizes
\[
\min_{\substack{A\subset \{a_0,\dots, a_{k-1}\} \\ |A|=k-r}} \sum_{a\in A}\nu_P (a_k-a).
\]
Again, one of the results of \cite{Bha09-1} is that the resulting sequences of exponents does not depend on the choice of $\{a_i\}$.

Recently, Grinberg and Petrov \cite{GP21} generalised the notion of $P$-orderings to the context of \emph{ultra triples}, which is a certain extension of ultrametric spaces, obtaining new proofs of several results of \cite{Bha97} and showing that all (prefixes of) $P$-orderings form a strong greedoid. A natural question which has been asked by Bhargava \cite{Bha-private} is then 
\begin{center}
{Do $r$-removed $P$-orderings also fit in the framework of \cite{GP21}?}
\end{center}

We resolve the question in the affirmative by showing that virtually all theorems of \cite{GP21} can be indeed generalised to cover the case of $r$-removed $P$-orderings. This also gives a new proof of \cite[Theorems 3, 4, 30]{Bha09-1}, see Corollary \ref{cor:removed-invariant} and Remark \ref{rem:Dedekind-domains}. 

Let us now introduce the objects we work with. An ultra triple $(E, w, d)$ is given by a ground set $E$, an arbitrary \emph{weight function} $w: E \rightarrow \mathbb{R}$, and a \emph{distance function} $d: \{(e, f ) \in E \times E \ | \ e  \neq f\}\rightarrow \mathbb{R}$ satisfying the ultrametric triangle inequality $d(a, b) \leq \max \{d(a, c), d(b, c)\}$. Note that unlike in the case of an ultrametric space we do not require distances to be non-negative. 

We consider an \emph{$r$-removed distance} from a point $v\in E$ to a finite subset $C\subset E$ defined by
\[
\dist_r(v, C):=
\begin{cases}
    \max\limits_{A\subset C, \, |A|=|C|-r} \sum_{x\in A} d(v, x)
    & |C|>r;\\
     0 & |C|\leq r.
\end{cases}
\]
In words, $\dist_r(v, C)$ is the maximum among sums of distances from $v$ to the points of $C$ except for some $r$ of them. Given this notion of a distance from a point to a set, we then consider \emph{greedy $r$-removed $m$-permutation} of a finite subset $C\subset E$ to be an ordered subset $(c_0,\dots, c_{m-1})$ of $C$ which is defined greedily by choosing elements $c_0, c_1,\dots$ one by one in such a way that for each $n\in \{0, 1, \dots, m-1\}$ the element $c_n$ maximizes
\[
w(c)+\dist_r(c, \{c_0,\dots, c_{n-1}\})
\]
over all possible $c\in C\setminus \{c_0,\dots, c_{n-1}\}$. A related concept is that of an \emph{$r$-removed perimeter} of a finite ordered set $A=(a_1,\dots a_n)$ which is defined to be 
\[
\per_r(A):=\sum_{a\in A} w(a) + \sum_{i=1}^n \dist_r(\{a_1, a_2, \ldots, a_{i-1}\}, a_{i}).
\]

We prove that $r$-removed perimeter does not depend on a permutation of a given set, see Lemma \ref{lemma:perm}, allowing to naturally extend the notion of an $r$-removed perimeter to finite (non-ordered) sets. Then, in Theorem \ref{thm:greedy-gives-max-per} we show that greedy $r$-removed $m$-permutation of a finite set $C$ has maximal $r$-removed perimeter among all subsets of $C$ of size $m$, and as a corollary we show that the sequence of numbers $w(c_j)+\dist_r(\{c_1, c_2, \ldots, c_{j-1}\}, c_j)$ does not depend on the choice of a greedy $r$-removed $m$-permutations. Then in Theorem \ref{thm:greedoid}  we show that all sets of maximal $r$-removed perimeter form a strong greedoid. At the end of the paper we also discuss some other notions of perimeter that fit into our framework. 

The rest of the paper is organised as follows. In Section \ref{sec:definitions} we formally define the objects we study and introduce a device we call a \emph{projection} which is later frequently used. Then, in Section \ref{theorems} we prove that $r$-removed perimeter of a set is well-defined and that construction of an greedy $r$-removed $m$-permutation gives the maximal $r$-removed perimeter among all subsets of size $m$. In Section \ref{sec:greedoid} we prove that subsets of maximal $r$-removed perimeter form a strong greedoid. Finally, Section \ref{sec:perimeter} is devoted to the discussion of a more general concept of a perimeter of a set, covered by our approach.

\section{Basic definitions and constructions} \label{sec:definitions}

We largely follow the notation used in \cite{GP21}, which we now briefly recall. Throughout the paper, we consider a set $E$ as our ground set, and refer to the elements of $E$ as points. For a non-negative integer $m$, an $m$-set means a subset $A$ of $E$ with $|A| = m$, and an $m$-permutation means an ordered set $A=(a_1,\dots, a_m)$ fromed by distinct elements of $E$. Analogously, if $B \subseteq E$ is a subset and $m$ is a non-negative integer, an $m$-subset of $B$ means an $m$-element subset of $B$ and an $m$-permutations of $B$ means an ordered set $A$ formed by $m$ distinct elements of $B$. The following definition already appeared in \cite{GP21}.

\begin{definition}
    An ultra triple is a triple $(E, w, d)$, where $E$ is a set, $w:R\rightarrow\mathbb{R}$ is an arbitrary \emph{weight function}, and $d:\{(e, f ) \in E \times E \ | \ e  \neq f\}\rightarrow \mathbb{R}$\footnote{One could also extend the  domain of $d$ to the whole of $E\times E$ by setting $d(a, a):=-\infty$ for all $a\in E$.} is a \emph{distance function} satisfying  
    \begin{itemize}
    \item $d(a,b) = d(b,a)$ for any two distinct $a,b \in E$.
    \item $
        d(a, b) \leq  \max \{d(a, c),\,  d(b, c)\} \text{ for any three distinct } a, b, c \in E.$
\end{itemize} 
\end{definition}
The inequality above is commonly known as the ultrametric triangle inequality; but unlike the distance function of an ultrametric space, we allow $d$ to take negative values. 
 
The following are formal definitions of the objects already mentioned in the introduction, namely, $r$-removed distance, $r$-removed perimeter, and an $r$-removed $m$-permutation.

\begin{definition}
    Let $(E, w, d)$ be an ultra triple, $C \subseteq E$ be a finite non-empty subset, and $v$ be any point in $E \setminus C$. We define $\dist_r (C, v)$ to be the maximum among all possible sums of distances from $v$ to some $|C|-r$ points of the set $C$. If $|C| \leq r$, we set $\dist_r (C, v) := 0$. 
\end{definition}
\begin{definition}
Let $(E, w, d)$ be an ultra triple. For a permutation $(a_1, a_2, \ldots, a_k)$ of a finite subset $A \subseteq E$, we define its $r$-removed perimeter by 
    $$ \per_r((a_1, a_2, \ldots, a_k)) := \sum_{a\in A} w(a) + \sum_{i=1}^k \dist_r(\{a_1, a_2, \ldots, a_{i-1}\}, a_{i}). $$
\end{definition}
\begin{definition} 
For an ultra triple $(E, w, d)$ let $C \subseteq E$ be a finite set, and $m$ be a non-negative integer.
A greedy $r$-removed $m$-permutation of $C$ is a list $(c_1, c_2, \ldots, c_m)$ of $m$ distinct elements of $C$ such that for each $i \in \{1, \ldots, m\}$ and each $x \in C \setminus \{c_1, c_2, \ldots, c_{i-1}\}$, we have 
\begin{equation}
\per_r((c_1,c_2,\ldots,c_i)) \geqslant \per_r((c_1,c_2,\ldots,c_{i-1},x)).
\end{equation}
\end{definition}

We now define a useful construction that we are going to use in the proofs, it earlier implicitly appeared in \cite{GP21}. We first recall the following definition from \cite{GP21}.

\begin{definition}  
Let $(E, w, d)$ be an ultra triple,
$A \subseteq E$ be a finite non-empty subset and $c \in E$ be any point. 
We define a subset $\proj_A (c)$ of $A$ as follows:
\begin{itemize}
    \item If $c \in A$, then $\proj_A (c):=\{c\}$;
    \item If $c \notin A$, then $\proj_A (c)$ is the set of all $a \in A$ that minimize the distance $d(c, a)$.
\end{itemize}
\end{definition}

Now we extend it in the following way 
\begin{definition}
Let $(E, w, d)$ be an ultra triple, $C=\{c_1,c_2,\ldots,c_k\}\subseteq E$ be a finite ordered set and $A$ be any $n$-subset of $E$, where $n \geq k$. We define $k$-permutation $(v_1,v_2,\ldots,v_k)$ of $A$ recursively as follows: $v_i$ is defined to be a projection of $c_i$ onto $A \setminus \{v_1,v_2,\ldots,v_{i-1}\}$  for each $i = 1,2,\ldots,k$. We denote $(v_1,v_2,\ldots,v_k)$ by $\proj(C\rightarrow A)$. These projections $v_i$ may be non-unique, in which case we take arbitrary element of the projection set.
\end{definition}
There are three important observations about these constructions that we now make. The first proposition already appeared in \cite[Lemma 2.13]{GP21} and we give its short proof for completeness. 
\begin{prop}
\label{prop:proj}
Let $(E, w, d)$ be an ultra triple and $A \subseteq E$ be a non-empty finite set. Then for a point $c \in E$, its projection $b\in \proj_A(c)$ and any $x\in A\setminus \{b\}$ we have $d(b,x) \leq d(c,x)$.
\end{prop}
\begin{proof}
    If $c \in A$ then $b=c$ and we trivially have an equality. Otherwise, since $x\in A$, by the definition of the projection 
    we have $d(c, x) \geq d(c, b)$ and so by the ultrametric triangle inequality we have 
    \[
    d(b, x)\leq \max{\{d(c, b), d(c, x)\}} = d(c, x).
    \]
\end{proof}

\begin{prop}
\label{prop:set-projection-1}
     Let $(E, w, d)$ be an ultra triple, $C=(c_1,c_2,\ldots,c_k)\subseteq E$ be a finite ordered set and $A$ be a $n$-subset of $E$, with $n \geq k$. Denote $\proj(C\rightarrow A)$ by $(v_1,v_2,\ldots,v_k) $. Then for each $j \in \{1,2,\ldots,k\}$ one has
\begin{equation*}
    (A\setminus \{v_1, \ldots, v_j\})\cap \{c_1,c_2,\ldots,c_j\} = \varnothing.
\end{equation*}
\end{prop}

\begin{proof}
Arguing by contradiction we assume for some $i\leq j\leq k$ that $c_i \in A\setminus \{v_1, \ldots, v_j\}$. This implies that $c_i\in A\setminus \{v_1,v_2,\ldots,v_{i-1}\}$. By definition this means that $v_i:=\proj_{A\setminus \{v_1,v_2,\ldots,v_{i-1}\}}(c_i) = c_i$. Hence, $v_i=c_i\in A\setminus \{v_1, \ldots, v_j\}$ which is impossible.
\end{proof}

\begin{prop}
\label{prop:proj_set}
Let $(E, w, d)$ be an ultra triple, $C=(c_1,c_2,\ldots,c_k)\subseteq E$ be a finite ordered set and $A$ be a $n$-subset of $E$, with $n > k$. Then for each $v\in A\setminus \proj(C\rightarrow A)$ 
\begin{equation*}
\dist_r(\proj(C\rightarrow A), v) \leq \dist_r(C, v).
\end{equation*}
\end{prop}
\begin{proof}
Denote $\proj(C\rightarrow A)$ by $(v_1,v_2,\ldots,v_k)$. The statement of the proposition would follow from the inequality $d(v_i, v)\leq d(c_i, v)$ for each $i\in\{1, 2, \dots, k\}$. But since 
$v \in A\setminus\{v_1,v_2,\ldots,v_{i-1}\}$ this inequality is given by Proposition \ref{prop:proj} applied to the set $A \setminus \{v_1,v_2,\ldots,v_{i-1}\}$ and points $c_i$, $v_i=\proj_{A \setminus \{v_1,v_2,\ldots,v_{i-1}\}}(c_i)$ and $v$.
\end{proof}

\section{Perimeter and greedy $r$-removed $m$-permutations} 
\label{theorems}

We first prove that any two permutations of a given set have the same $r$-removed perimeter.
\begin{lemma} \label{lemma:perm}
Any two permutations of a finite set $A\subseteq E$ have the same $r$-removed perimeter.
\end{lemma}

\begin{proof}
It suffices to prove the statement for pairs of permutations which differ by one transposition. The general case is reduced to it by consecutive transpositions.

 Let us prove the statement for permutations $(a_1, \ldots , a_t, a_{t+1}, \ldots a_k)$ and $(a_1, \ldots , a_{t-1}, a_{t+1}, a_t, \ldots , a_k)$. Denote by $C$ the set $\{a_1, a_2, \ldots a_{t-1}\}$. Many summands from the definition of $r$-removed perimeter coincide, all that remains to prove is
\begin{equation}
\label{eq:dist-permutation-condition}
    \dist_r(C, a_t) + \dist_r(C\cup \{a_t\} , a_{t+1}) = \dist_r(C, a_{t+1}) + \dist_r(C\cup \{a_{t+1}\}, a_t).
\end{equation}

If $t \leq r$, both sides are 0. Otherwise, we let $z = d(a_{t}, a_{t+1})$, $x_j = d(a_{t}, a_j)$ and $y_j = d(a_{t+1}, a_j)$, where $j=1, \ldots, t - 1$. In what follows we only consider triangles of the form $a_{t}a_{t+1}a_j$ for some $j=1,2,\ldots, t-1$ and use ultrametric triangle inequality for them. 

We colour triangles with two sides strictly greater than $z$ in red, in which case $x_j = y_j$ by the ultrametric inequality. Triangles coloured in red correspond to some largest distances from points $a_t$ and $a_{t+1}$ to the set $C$ which coincide. In any other triangle we must have $x_i = z \geqslant y_i$ or $y_i = z\geqslant x_i$.

If there are at least $t - r$ red triangles, then $\dist_r(C, a_t) = \dist_r(C, a_{t+1})$, ${\dist_r(C\cup \{a_t\} , a_{t+1}) = \dist_r(C\cup \{a_{t+1}\}, a_t)}$ and \eqref{eq:dist-permutation-condition} is true. 

If there are less than $t-r$ red triangles, then $\dist_r(C\cup \{a_t\} , a_{t+1}) = \dist_r(C , a_{t+1}) + z$ and $\dist_r(C\cup \{a_{t+1}\}, a_t) = \dist_r(C, a_t) + z$. By substituting these expressions into \eqref{eq:dist-permutation-condition}, we again get an equality.
\end{proof}

In  light of this lemma we have the following definition

\begin{definition}
For a finite subset $A \subseteq E$, we define its $r$-removed perimeter $\per_r (A)$ to be the common $r$-removed perimeter of all permutations of $A$.
\end{definition}

\begin{remark}
For the case $r=0$, the $r$-removed perimeter is the sum of the distances between all unordered pairs of points plus the sum of the weight function of all points. This case was considered in \cite{GP21}.
\end{remark}

\begin{theorem} 
\label{thm:greedy-gives-max-per}
Let $(E, w, d)$ be an ultra triple and $C \subseteq E$ be a finite subset, and  $m$ and $r$ be non-negative integers. Let $(c_1, c_2, \ldots, c_m)$ be any greedy $r$-removed $m$-permutation of $C$. Then, for each $k \in \{0,1,...,m\}$, the set $\{c_1,c_2,\ldots,c_k\}$ has maximum $r$-removed perimeter among all $k$-subsets of $C$.
\end{theorem}

\begin{proof}

Given a greedy $r$-removed $m$-permutation $(c_1,c_2,\ldots,c_m)$, we want to prove that for any $k$-subset $A\subseteq C$, $\per_r(A) \leq \per_r (\{c_1, c_2, \ldots , c_k\})$. We induct on $k$. 

For $k=0$ both perimeters are 0, and so the inequality is trivially true. For the induction step from $k - 1$ to $k$,
let $(v_1,v_2,\ldots,v_k):=\proj(\{c_1,c_2,\ldots,c_k\}\rightarrow A)$, which is an ordering of $A$. Then by Proposition \ref{prop:set-projection-1},
\begin{equation*}
    v_{k} \notin \{c_1,c_2,\ldots,c_{k-1}\}.
\end{equation*}

By induction hypothesis we know that
\[
\per_r(\{v_{1}, v_{2}, \ldots, v_{k-1}\}) \leq \per_r(\{c_{1}, c_{2}, \ldots, c_{k-1}\}),
\]
and so to complete the induction step it suffices to show that 
\[
\per_r(A) - \per_r(\{v_{1}, v_{2}, \ldots, v_{k-1}\}) \leq \per_r (\{c_1, c_2, \ldots , c_k\}) - \per_r(\{c_{1}, c_{2}, \ldots, c_{k-1}\}),
\] 
which, implicitly using Lemma \ref{lemma:perm}, can be equivalently written as 
\begin{equation}
\label{ineq:greedy-gives-max-per-thm-tocheck}
    w\left(v_{k}\right)+\dist_r(\{v_1, v_2, \ldots, v_{k-1}\}, v_k) \leq w\left(c_{k}\right)+\dist_r(\{c_1, c_2, \ldots, c_{k-1}\}, c_k).
\end{equation}

We now turn to proving \eqref{ineq:greedy-gives-max-per-thm-tocheck}. Since $v_k \in A \setminus \{c_1,c_2,\ldots,c_{k-1}\} \subseteq C \setminus \{c_1,c_2,\ldots ,c_{k-1}\}$ (recall that $A \subseteq C$), we have $\per_r \{c_1,c_2,\ldots,c_{k-1},v_k\} \leq \per_r\{c_1,c_2,\ldots,c_k\}$ by the definition of a greedy $r$-removed $m$-permutation. Subtracting $\per_r(\left\{c_{1}, c_{2}, \ldots, c_{k-1}\right\})$ from both sides we arrive at 
\begin{equation*}
    w\left(v_{k}\right)+\dist_r(\{c_1, c_2, \ldots ,c_{k-1}\}, v_k) \leq w\left(c_{k}\right)+\dist_r(\{c_1, c_2, \ldots, c_{k-1}\}, c_k)
\end{equation*}

And so to deduce \eqref{ineq:greedy-gives-max-per-thm-tocheck} it remains to show that 
\begin{equation*}
      \dist_r(\{v_1, v_2, \ldots ,v_{k-1}\}, v_k) \leq \dist_r(\{c_1, c_2, \ldots ,c_{k-1}\}, v_k).
\end{equation*}
Which is nothing else but the statement of Proposition \ref{prop:proj_set} for the point $v=v_k$ and  $\{v_1, v_2, \ldots ,v_{k-1}\} = \proj (\{c_1, c_2, \ldots ,c_{k-1}\}\rightarrow A)$.
\end{proof}

\begin{remark}\label{rem:perimeter-equality}
It follows from the proof that if the equality 
\[
\per_r(\{v_1,v_2,\ldots,v_k\}) = \per_r(\{c_{1}, c_{2}, \ldots, c_{k}\})
\] 
holds for $(v_1,v_2,\ldots,v_k):=\proj(\{c_1,c_2,\ldots,c_k\}\rightarrow A)$, then for each $j<k$ one also has an equality $\per_r(\{v_{1}, v_{2}, \ldots, v_j\}) = \per_r(\{c_{1}, c_{2}, \ldots, c_j\})$.
\end{remark}

\begin{corollary}
\label{cor:removed-invariant}
 Let $C\subseteq E$ be a set, $m$ and $r$ be a non-negative integers, $j\in \{ 1, 2, \ldots m\}$. If $(c_1, c_2, \ldots, c_m)$ is a greedy $r$-removed $m$-permutation of $C$, then the number
\begin{align*}
 w\left(c_{j}\right)+\dist_r(\{c_1, c_2, \ldots, c_{j-1}\}, c_j) 
\end{align*}
does not depend on the choice of this greedy $r$-removed $m$-permutation but only depends on $C, r$ and $j$.
\end{corollary}
\begin{proof}
 By Theorem \ref{thm:greedy-gives-max-per}, for each $k\leq m$ the set $\{c_1, c_2, \ldots, c_k\}$ has maximum perimeter among all $k$-subsets of $C$, which implies that $\per_r(\{c_1, c_2, \ldots, c_k\})$
 does not depend on the choice of the greedy $r$-removed $m$-permutation of $C$. It remains to note that 
\begin{align*}
     w(c_j) + \dist_r(\{c_1, c_2, \ldots, c_{j-1}\}, c_j)
     =
     \per_r(\{c_1, c_2, \ldots, c_j\}) - \per_r(\{c_1, c_2, \ldots, c_{j-1}\}).
\end{align*}
\end{proof}

\begin{remark}\label{rem:Dedekind-domains}
    As a special case of this corollary we obtain the results of \cite[Theorems 3, 4, 30]{Bha09-1}. Indeed, for a Dedekind domain $D$, a prime ideal $P\subset D$, and $h\in \mathbb{Z}_{\geq 0}$, the distance function $d_{P, h}(a, b):=-\max(h, \nu_P(a-b))$ satisfies the ultrametric triangle inequality and so the result follows from  Corollary \ref{cor:removed-invariant} applied to an ultra triple $(S, w\equiv 0, d_{P, h})$.
\end{remark}

We now prove the converse of Theorem \ref{thm:greedy-gives-max-per}, namely, that any set of maximal $r$-removed perimeter is a prefix of some greedy $r$-removed $m$-permutation.

\begin{theorem}
\label{thm:any-set-with-max-per-is-greedy}
Let $(E, w, d)$ be an ultra triple, $C \subseteq E$ be a finite set, and  $m$ be a non-negative integer such that $|C| \geqslant m$. For $k \in \{0,1,\ldots,m\}$ let  $A$ be a $k$-subset of $C$ having maximum $r$-removed perimeter (among all $k$-subsets of $C$). Then, there exists a greedy $r$-removed $m$-permutation of $C$ for which $A$ is a prefix of this permutation.
\end{theorem}

\begin{proof}
Choose an arbitrary greedy $r$-removed $m$-permutation $(c_1, c_2, \ldots, c_m)$ of $C$ by starting with any point and continuing  the sequence greedily choosing elements from the remaining points. By Theorem \ref{thm:greedy-gives-max-per}, the set $(c_1, c_2, \ldots, c_k)$ has maximum perimeter among all $k$-subsets of $C$. Hence, $\per_r(A) =  \per_r(\left\{c_{1}, c_{2}, \ldots, c_{k}\right\})$ since
the set $A$ also has maximum $r$-removed perimeter among them.

Let $(v_1, v_2, \ldots, v_k):=\proj(\{c_1, c_2, \ldots , c_k\} \rightarrow A)$. What we want to prove is that there exists a greedy $r$-removed $m$ permutation of $C$ which starts from $(v_1, v_2,\dots, v_k)$, which is equivalent to checking that for each $p\leq k$ the point $v_p$ maximizes 
\[
w(x)+\dist_r(\{v_1,\dots,v_{p-1}\}, x)
\]
over all $x\in C\setminus \{v_1,\dots,v_{p-1}\}$.

As mentioned in Remark \ref{rem:perimeter-equality}, the fact that $\per_r(\{v_1,v_2,\ldots,v_k\})=\per_r(\{c_1, c_2, \ldots , c_k\})$ implies that for each $j\leq k$ we have $\per_r(\{v_1,v_2,\ldots,v_j\})=\per_r(\{c_1, c_2, \ldots , c_j\})$. In particular, this holds for $j=p$. Now, arguing by contradiction we assume that there exists $x\in C\setminus \{v_1,\ldots,v_{p-1}\}$ such that 
\[
w(x)+\dist_r(\{v_1,\ldots,v_{p-1}\}, x)
> 
w(v_p)+\dist_r(\{v_1,\ldots,v_{p-1}\}, v_p).
\]
This would mean that 
\[
\per_r(\{v_1,\ldots, v_{p-1}, x\})> \per_r(\{v_1,\ldots, v_{p-1}, v_p\}) = 
\per_r(\{c_1,\ldots, c_{p-1}, c_p\}),
\]
contradicting the fact that $\{c_1,\ldots, c_{p-1}, c_p\}$ has the largest $r$-removed perimeter among all subsets of $C$ of size $p$.
\end{proof}

\section{Strong greedoid of maximum perimeter sets}\label{sec:greedoid}

In \cite{GP21} it was shown that sets maximizing the perimeter (i.e. $r$-removed perimeter with $r=0$) form a strong greedoid. In this section we generalize this statement to all $r\geq 0$. We start by recalling the relevant definitions from the theory of greedoids.

\begin{definition}
A collection $\mathcal{F} \subseteq 2^E$ of subsets of a finite set $E$ is called a \textbf{greedoid} (on the ground set $E$) if it satisfies the following three axioms:
\begin{enumerate}[label=(\roman*)]
    \item $\varnothing \in \mathcal{F}$. \label{greedoid:1}
    \item If $A \in \mathcal{F}$ satisfies $|A| > 0$, then there exists $a \in A$ such that $A \setminus  a \in \mathcal{F}$. \label{greedoid:2}
    \item If $A, B \in \mathcal{F}$ satisfy $|A| = |B| + 1$, then there exists $a \in A\setminus B$ such that $B \cup a \in \mathcal{F}$. \label{greedoid:3}
\end{enumerate}

A greedoid $\mathcal{F}$ on a ground set $E$ is called a \textbf{strong greedoid} (also known as ``Gauss greedoid") if it additionally satisfies the following axiom:
\begin{enumerate}[label=(\roman*)]
\setcounter{enumi}{3}
\item If $A,B \in \mathcal{F}$ satisfy $|A| = |B| + 1$, then there exists $a \in A \setminus B$ such that $B \cup a \in \mathcal{F}$ and $A \setminus a \in \mathcal{F}$. \label{greedoid:4}
\end{enumerate}
\end{definition}

There are several equivalent definitions of a greedoid in the literature, ours is taken from \cite[Section IV.1]{KLS91}. Specifically, our axioms (i) and (iii) align with conditions (1.4) and (1.6) in \cite[Section IV.1]{KLS91}, while axioms (i) and (ii) establish $(E,\mathcal{F})$ as an accessible set system. The definition of a strong greedoid can be found in \cite{BS99}. 

Now we assume that the set $E$ is finite. The following theorem shows that sets with maximal $r$-removed perimeter form a strong greedoid.

\begin{theorem}
\label{thm:greedoid}
Let $(E, w, d)$ be an ultra triple on a finite ground set and $\mathcal{F}_r$ denote the collection of subsets $A \subseteq E$ that have maximum $r$-removed perimeter among all $|A|$-sets: 
$$\mathcal{F}_r:=\{A\subseteq E \ | \ \per_r(A)\geq \per_r(B) \text{ for all } B \subseteq E \text{ satisfying } |B|=|A|\}.$$ 
Then $\mathcal{F}_r$ is a strong greedoid on the ground set $E$.
\end{theorem}

We start by proving the following lemma.

\begin{lemma}
\label{lemma:greedoid}
Let $A$ and $B$ be two subsets of $E$ such that $|A|=|B|+1$. Then, there exists $u \in A \setminus B$ satisfying
\begin{equation}\label{eq:swap}
\per_r (A \setminus u) + \per_r (B \cup  u) \geq \per_r (A) + \per_r (B)
\end{equation}
\end{lemma}
\begin{proof}
Let $k = |B|$ and so $|A| = k + 1$. With a slight abuse of notation we denote by $B$ an arbitrary ordered set formed by elements of $B$, which we fix from now on. Define $(v_1, v_2,\ldots,v_k):=\proj (B\rightarrow A)$ and let $u$ be the unique element of $A\setminus \{v_1, v_2, \ldots, v_k\}$. By Lemma \ref{prop:set-projection-1} we have $u\notin B$. We now want to prove \eqref{eq:swap} for this choice of $u$. Subtructing $\per_r (A \setminus u) + \per_r (B) + w(u)$ from both sides we arrive at an equivalent inequality 

$$\dist_r(B, u) \leq \dist_r(A\setminus u, u),$$
which is simply the result of Proposition \ref{prop:proj_set} applied to $u$ and $A\setminus u = \proj (B \rightarrow A)$.
\end{proof}

\begin{proof}[Proof of Theorem \ref{thm:greedoid}]

First note that property (i) is trivial, and (iii) immediately follows from (iv). Furthermore, since $E$ is finite, for any $s\leq |E|$ there exists $B\in \mathcal{F}_r$ with $|B|=s$, and so by choosing arbitrary $B\in \mathcal{F}_r$ with $|B|=|A|-1$ we can deduce (ii) from (iv).

To prove (iv) we use Lemma \ref{lemma:greedoid} to construct $u\in A\setminus B$ satisfying \eqref{eq:swap}. Since $A\in \mathcal{F}_r$ we must have $\per_r (B \cup  u) \leq \per_r (A)$. Similarly, $B\in \mathcal{F}_r$ implies $\per_r (A \setminus u) \leq \per_r(B)$. Together with \eqref{eq:swap} these two inequalities immediately imply that 
\[
\per_r (A \setminus u) = \per_r(B),
\qquad
\per_r (B \cup  u) = \per_r (A),
\]
which means that both $A\setminus u$ and $B\cup u$ are in $\mathcal{F}_r$. This shows that $\mathcal{F}_r$ is a strong greedoid.
\end{proof}

\section{Other perimeters} \label{sec:perimeter}
Instead of the $r$-removed distance $\dist_r(C, v)$ we could start from some other notion of a distance from a point to a set, call it $\dist(C, v)$, and define $\per(\{a_1,\dots, a_n\})$ of an ordered set $A:=(a_1,\dots, a_n)$ by setting $\per(A) := \sum_{a\in A} w(a) + \sum_{i=1}^k \dist(\{a_1, a_2, \ldots, a_{i-1}\}, a_{i})$. Tracking the proofs of Theorems \ref{thm:greedy-gives-max-per}, \ref{thm:any-set-with-max-per-is-greedy} and \ref{thm:greedoid} we see that the only two properties of the $\dist$ functions that we use are given by the following

\begin{definition}\label{def:property-S}
 We say that a function $\dist$ satisfies property $\mathbf{S}$ if for any set $\left\{ c_{1}, c_{2}, \ldots, c_{n}\right\} =: C \subseteq E$ and $x, y \in E \setminus C$ one has
\begin{enumerate}[label=\textbf{(S\arabic*)}]
    \item $\dist(C, x) + \dist(C\cup \{x\} , y) = \dist(C, y) + \dist(C\cup \{y\}, x)$;
        
    \item  If $d(c_i, x) \leq d(c_i, y)$ for each $i \in \{1, 2, \ldots, n\}$, then $$\dist(\{c_1,c_2,\ldots,c_{n}\}, x) \leq \dist(\{c_1,c_2,\ldots,c_n\}, y).$$
\end{enumerate}    
\end{definition}

\begin{remark}
    Indeed, property $\mathbf{(S1)}$ is used in Lemma \ref{lemma:perm} to prove that the perimeter of a set is well-defined, and is, in fact, equivalent to this lemma. Property $\mathbf{(S2)}$ is used in the proofs of Theorems \ref{thm:greedy-gives-max-per} and \ref{thm:any-set-with-max-per-is-greedy}.
\end{remark}

We now give a large family of distances satisfying property $\mathbf{S}$.

\begin{lemma}\label{lemma:general-perimeter}
Let $(E, w, d)$ be an ultra triple and $f=(f_n)_{n=1}^{\infty}$ be a sequence of non-decreasing functions. 
For a set $\{c_1, \ldots, c_n\} = C \subseteq E$ and $x \in E\setminus C$, let $\dist_f(C, x) := \sum_{i=1}^n f_i(d_i)$, where $(d_1, d_2, \ldots, d_n)$ is the ordered set of values $d(c_1, x), \ldots, d(c_n, x)$ arranged in non-decreasing order. Then  $\dist_f$ satisfies property $\mathbf{S}$.
\end{lemma}

\begin{proof}
We first check $\mathbf{(S1)}$. Since the property is linear in $f=(f_n)_{n=1}^\infty$, it suffices to check it for $\dist_f$ with 
\[
f_j:=
\begin{cases}
0 & j\leq r;\\
g & j > r.
\end{cases}
\]
Where $g:\mathbb{R}\rightarrow \mathbb{R}$ is some fixed non-decreasing function. Indeed, one easily sees that the linear span of these sequences contains the whole cone of possible sequences of non-decreasing functions. But for this specific choice of $f$, the distance $\dist_f$ is nothing else but the $r$-removed distance for the ultra triple $(E, w, d_g)$ where $d_g$ is given by $d_g(a, b):=g(d(a, b))$\footnote{One easily sees that $d_g$ satisfies the ultemetric inequality for any non-decreasing $g$.} and so the equality $\textbf{(S1)}$ is given by \eqref{eq:dist-permutation-condition} from the proof of Lemma \ref{lemma:perm}. 

To check $\textbf{(S2)}$ it suffices to note that if one $n$-tuple is point-wise smaller than another, then the same remains true after each of the $n$-tuples is sorted from the smallest value to the largest, and then use the fact that each $f_j$ is non-decreasing.
\end{proof}

For sufficiently large spaces and under certain natural conditions on the distance function $\dist$ from a point to a set we manage to prove the reverse of Lemma \ref{lemma:general-perimeter}. To avoid stating technical conditions we prove the result for the space $(\mathbb{Z}, w, -\nu_p)$ in which distance between points $a, b\in \mathbb{Z}$ is given by $-\nu_p(a-b)$, where $\nu_p$ stands for the $p$-adic valuation and we further assume that $p>2$. 

\begin{lemma}
    Consider an ultra triple $(\mathbb{Z}, w, -\nu_p)$ with arbitrary weight function $w$ and $p>2$. Assume that the value of the distance function $\dist(C, x)$ from a point $x$ to a set $C$ depends only on the set of distances from $x$ to the points of $C$ and that $\dist$ satisfies property $\mathbf{S}$. Then $\dist\equiv \dist_f$ for some sequence of non-decreasing functions $f=(f_n)_{n=1}^\infty$.
\end{lemma}
\begin{proof}
    By assumption there exists a sequence of symmetric functions $g_n:\mathbb{Z}_{\leq 0}^n \rightarrow\mathbb{R}$ indexed by $n\geq 1$, such that for any point $x$ and any set $\{c_1,\dots, c_n\}$ not containing $x$ we have
    \[
    \dist(x, \{c_1,\dots, c_n\})=g_n(d(c_1, x), \dots, d(c_n, x)).
    \]
    We want to prove the existence of a sequence of non-decreasing functions $(f_n)_{n=1}^\infty$ such that for any non-positive integers $d_1\leq d_2\leq\dots\leq d_n$
    \begin{equation}\label{eq:g-f-correspondence}
    g_n(d_1,d_2,\dots d_n) = \sum_{j=1}^n f_j(d_j).
    \end{equation}
    We prove the existence of functions $f_j$ by induction on $j$, and for the base case we set $f_1:=g_1$ which is non-decreasing by the second condition of property $\mathbf{S}$.

    Now assume that $f_1,\dots, f_{m-1}$ are already defined in such a way that \eqref{eq:g-f-correspondence} is satisfied for all $n<m$ and we want to define $f_m$. For each $d\in\mathbb{Z}_{\leq 0}$ we set
    \[
    f_m(d):=g_m(d, d, \dots, d)-g_{m-1}(d, d, \dots, d),
    \]
    where we have $m$ arguments equal to $d$ in the first case and $m-1$ arguments equal to $d$ in the second. 
    
    First, we check that \eqref{eq:g-f-correspondence} is satisfied for $n=m$. For this, given non-positive integers $d_1\leq d_2\leq\dots\leq d_n$ we consider two points $x, y\in\mathbb{Z}$ with $d(x, y)=d_n$ and a set of points $C:=\{c_1,\dots, c_{n-1}\}$ such that for each $j=1,\dots, n-1$ we have  $d_j:=d(c_j, x)$ and $d(c_j, y)=d_n$. The existence of such a set follows from the property of $\mathbb{Z}$ with the $p$-adic distance (where $p>2$) which guarantees that for any two points $a, b\in\mathbb{Z}$ and any $\ell \leq d(a, b)$ there exists $c\in\mathbb{Z}$ such that $d(a, c)=\ell$ and $d(b, c) = d(a, b)$. Using $\mathbf{(S1)}$ we write 
    \begin{align*}
        g(d_1,\dots, d_n)=\dist(C\cup \{y\}, x) = \dist(C, x) + \left(\dist(C\cup \{x\}, y) - \dist(C, y)\right),
    \end{align*}
    and it remains to observe that the difference in parentheses is equal to $f_n(d_n)$ by the definition of $f_n$ and 
    \[
    \dist(C, x) = \sum_{j=1}^{n-1} f_j(d_j)
    \]
    by induction hypothesis. 
    
Second, we show that $f_m$ is non-decreasing. Let $\ell_1, \ell_2\in\mathbb{Z}_{\leq 0}$ satisfy $\ell_1\leq \ell_2$. By \eqref{eq:g-f-correspondence} we have 
\[
f_m(\ell_2)-f_m(\ell_1)=g_m(\ell_1,\dots, \ell_1, \ell_2)-g_{m}(\ell_1, \dots, \ell_1, \ell_1)
\]
and so $f_m(\ell_2)\geq f_m(\ell_1)$ directly follows from $\mathbf{(S2)}$.
\end{proof}

\textbf{Acknowledgements:} We would like to thank Fedor Petrov for his guidance in this project. 

\bibliographystyle{plain}

\end{document}